\documentclass[12pt]{amsart}

\usepackage{graphicx, verbatim,amsmath,amssymb}
\newcommand{\Om}{{\Omega}}

\newcommand{\aaa}{{\mathcal A}}

\newcommand{\ttt}{\mathbb T}

\newcommand{\cM}{{\mathcal M}}
\newcommand{\T}{{\mathcal T}}
\newcommand{\R}{{\mathbb R}}
\newcommand{\Q}{{\mathbb Q}}

\newcommand{\Torus}{{\mathbb T}}
\newcommand{\Z}{{\mathbb Z}}

\newtheorem{thm}{Theorem}[section]
\newtheorem*{thm*}{Theorem}

\newtheorem{con}[thm]{Conjecture}
\newtheorem{lem}[thm]{Lemma}

\newtheorem{dfn}[thm]{Definition}

\everymath{\displaystyle}

\theoremstyle{definition}
\newtheorem{ex}{Example}

\input xy
\xyoption {all}

\begin{document}
\title{Small cocycles, fine torus fibrations,
and a $\Z^2$ subshift with neither}

\author{Alex Clark}
\address{Alex Clark, Department of Mathematics, University of Leicester, University Road, Leicester LE1 7RH, United Kingdom}
\email{Alex.Clark@le.ac.uk}

\author{Lorenzo Sadun}
\address{Lorenzo Sadun, Department of Mathematics, University of Texas,
Austin, TX 78712, USA}
\email{sadun@math.utexas.edu}
\date{\today}
\begin{abstract}
  Following an earlier similar conjecture of Kellendonk and Putnam, Giordano, Putnam and Skau conjectured that all minimal, free $\Z^d$
  actions on Cantor sets admit ``small cocycles.''   These represent
  classes in $H^1$ that are mapped to small vectors in $\R^d$ by the
  Ruelle-Sullivan (RS) map. We show that there exist $\Z^2$ actions
  where no such small cocycles exist, and where the image of $H^1$
  under RS is $\Z^2$.  Our methods involve tiling spaces and shape
  deformations, and along the way we prove a relation between the
  image of RS and the set of ``virtual eigenvalues'', i.e. elements of
  $\R^d$ that become topological eigenvalues of the tiling flow after
  an arbitrarily small change in the shapes and sizes of the tiles.
\end{abstract}

\maketitle

\setlength{\baselineskip}{.6cm}

\section{Introduction and statement of results}

In this paper we consider cohomological properties of minimal, free
$\Z^d$ actions on Cantor sets. In particular, we consider the first
group cohomology of a $\Z^d$ action, and its image under the
Ruelle-Sullivan map. Kellendonk and Putnam~\cite[Conj. 16]{KP}
conjectured (under the additional assumption of unique ergodicity of
the action) that the image of the Ruelle-Sullivan map is always dense
in the dual space $(\R^d)^*$ to $\R^d.$ Giordano, Putnam and Skau
\cite{GPSCocycle} conjectured (without the unique ergodicity
assumption) that the image of this map is always dense in the dual
space $(\R^d)^*$ to $\R^d$, and further that there always exist
``small, positive cocycles''.  Giordano, Putnam and Skau
\cite{GPSCocycle} showed that a number of interesting consequences
would follow from these conjectures, including the existence of a
free, minimal $\Z$ action that is not orbit-equivalent to any $\Z^2$
action.

In this paper we show that all these conjectures are in fact incorrect.

\begin{thm}[Theorem \ref{mainthm}] \label{nogo} There exists a free,
  minimal and uniquely ergodic $\Z^2$ subshift that does not admit any
  small cocycles, for which the image of $H^1$ under the
  Ruelle-Sullivan map is merely the natural $\Z^2$ that comes from
  constant cochains.
\end{thm}

In fact, $H^1$ of this $\Z^2$ subshift is itself equal to $\Z^2$
(Theorem \ref{icing-on-the-cake}).

Before developing this counterexample, we consider the significance of
the image of the Ruelle-Sullivan map. Here we work in the setting of
$\R^d$ actions on tiling spaces. $\Z^d$ actions on Cantor sets,
subshifts, and tilings with finite local complexity (FLC) are closely
related. Every expansive $\Z^d$ action can be realized as a subshift,
the suspension of the $\Z^d$ action on a subshift is a tiling space
with FLC, and every tiling space with FLC is homeomorphic to the
suspension of a subshift \cite{SW}.  As a result, topological theorems
about each of these categories can give important insights into the
other two.

We then relate the first cohomology of a tiling space to spectral
theory, and to realizations of that tiling space as a Cantor bundle
over a torus.

\begin{dfn} If $\Omega$ is an aperiodic tiling space with FLC, and if $\lambda
  \in (\R^d)^*$, we say that $\lambda$ is a {\bf virtual eigenvalue}
  of the natural $\R^d$ action if there exist arbitrarily small
  changes to the shapes and sizes of the tiles (aka arbitrarily small
  time changes) such that $\lambda$ is a topological eigenvalue of the
  resulting $\R^d$ actions.
\end{dfn}

Shape changes can also be used to make the translation dynamics
topologically conjugate to the natural translation action on a Cantor
bundle over a torus. (Henceforth, all Cantor bundles over tori will
be assumed to carry this action.)
Indeed, this is how the authors of \cite{SW}
showed that all FLC tiling spaces are homeomorphic to suspensions of
subshifts.  In this paper we consider which tori can be bases of such
bundles after arbitrarily small shape changes. For uniquely ergodic
tiling spaces, the answer is especially simple:

\begin{thm} \label{bigtori} Let $\Omega$ be a uniquely ergodic  FLC
  tiling space whose natural $\R^d$ action is minimal and free,
and let $\lambda \in
  (\R^d)^*$. Then the following conditions are equivalent:
\begin{enumerate}
\item $\lambda$ is in the closure of the image of $H^1(\Omega)$ under the
  Ruelle-Sullivan map.
\item $\lambda$ is a virtual eigenvalue.
\item There is an arbitrarily small shape change that transforms
  $\Omega$ into a Cantor bundle over a torus $\R^d/L$, such that $\lambda$ is a
  period of the dual torus $(\R^d)^*/L^*$. (Here $L \subset \R^d$ is a lattice
and $L^* \subset (\R^d)^*$ is the dual lattice.)
\end{enumerate}
\end{thm}

Similar results apply to linearly independent sets of virtual
eigenvalues, and in particular to bases of eigenvalues.

\begin{thm} \label{morebigtori} Let $\Omega$ be a uniquely ergodic FLC
  tiling space whose natural $\R^d$ action is minimal and free,
and let $(\lambda_1,
  \ldots, \lambda_d)$ be a basis for $(\R^d)^*$. Let $L'= \lambda_1 \Z
  \oplus \cdots \oplus \lambda_d \Z$ be the lattice spanned by the
  $\lambda_i$'s, dual to a lattice $L \subset \R^d$, and let $\ttt =
  \R^d/L$. Then the following are equivalent:
 \begin{enumerate}
 \item All of the $\lambda_i$'s are virtual eigenvalues.
 \item There exist arbitrarily small shape changes that convert
   $\Omega$ to a Cantor bundle over $\ttt$.
 \end{enumerate}
 \end{thm}

The following is then an immediate corollary:
\begin{thm} \label{bigsquares} Let $\Omega$ be a uniquely ergodic FLC
  tiling space with a minimal, free $\R^d$ action. The image of
  $H^1(\Omega)$ is dense if and only if, for arbitrary length
  $R$, we can apply an arbitrarily small shape change to convert
  $\Omega$ into a Cantor bundle over $\R^d/(R \Z)^d$.
\end{thm}

It is worth contrasting Theorem \ref{bigtori} with Theorem 3.9 of
\cite{GPSCocycle}. The latter theorem states that the existence of
small positive cocycles implies that, for arbitrary length $R$,
one can break up each tiling into locally defined regions such that
each region contains a cube of side $R$.  This in turn gives an easy
orbit equivalence between the original $\Z^d$ action on a Cantor set
and a $\Z$ action.

However, the converse is not true.  The counterexample we construct to
the conjectures of Giordano, Kellendonk, Putnam and Skau does not admit small
positive cocycles.  However, it is built as a hierarchical tiling
space, and so the tilings do admit partitions into arbitrarily large
locally defined rectangular and square regions.  By contrast to
Theorem 3.9 of \cite{GPSCocycle}, Theorems \ref{bigtori} and
\ref{morebigtori} and \ref{bigsquares} are ``if and only if''
statements.  Since our counterexample does not have elements of
$H^1(\Omega)$ whose images under Ruelle-Sullivan are small, it does
not admit large torus fibrations, and does not have any virtual
eigenvalues beyond $\Z^2$.

\begin{dfn}
  For a $\Z^2$-subshift $\Xi$, let $\phi^{(n_1,n_2)}$ represent
  translation by $(n_1,n_2) \in \Z^2$. For each integer $N$, let $N_+$
  denote the integers greater than $N$, and let $N_-$ denote the
  integers less than $N$.  We say that $\Xi$ {\bf admits a horizontal
    shear} if there exists an element $u$ of the subshift space
  and integers $N, N'$ such that, for every integer $n$, there is an
  element of $\Xi$ that agrees with $u$ on $\Z \times N_+$ , and
  agrees with $\phi^{(n,0)} u$ on $\Z \times N'_-$.
\end{dfn}
Vertical shears are defined similarly.  Shears have a profound effect
on the topological eigenvalues.

\begin{thm} [Theorem \ref{baby-shears}] \label{no-top-evals}
In a minimal
$\Z^2$ subshift that admits
shears in both coordinate directions, the spectrum of
topological eigenvalues is precisely $\Z^2$.
\end{thm}

We conjecture that this constraint on topological eigenvalues extends, upon
perturbation, to a constraint on virtual eigenvalues.

\begin{con} \label{shears} Let $\Xi$ be a minimal, aperiodic and uniquely ergodic
  $\Z^2$ subshift, and let $\Omega$ be the suspension of $\Xi$.  Let
  $\lambda = (\lambda_x,\lambda_y)$ be a virtual eigenvalue of
  $\Omega$.
\begin{enumerate}
\item If $\Xi$ admits a horizontal shear, then $\lambda_x \in \Z$.
\item If $\Xi$ admits a vertical shear, then $\lambda_y \in \Z$.
\item If $\Xi$ admits both a horizontal shear and a vertical shear, then $\lambda \in \Z^2$.
\end{enumerate}
\end{con}

Given Conjecture \ref{shears}, proving Theorem \ref{nogo} would reduce
to exhibiting a $\Z^2$ subshift with shears in both directions. Such
subshifts are already known. A particularly nice example, Natalie Frank's
non-Pisot Direct Product Variation (DPV) tiling, was
developed in \cite{Frank-primer, FR} and further explored in
\cite{FrankSadunGD, FrankSadunTP}.

In fact, we are able to use the shear properties of the Frank
DPV to prove that $RS(H^1(\Omega))=\Z^2$.
This proves Theorem \ref{nogo} directly, without relying on
Conjecture \ref{shears}.  Indeed, the methods of this proof generalize
to a wide class of DPV tilings, providing evidence for Conjecture
\ref{shears}.

In a related vein, we consider bundle structures and the return
dynamics induced by them, and how this depends on the (non)existence of
small cocycles. These results are more technical, and we leave a precise
statement of the theorems to sections \ref{finest} and \ref{return}.

In section \ref{finest} we consider the question of
when tiling spaces (or subshifts) admit optimal ``finest'' torus
fibrations. The results depend on whether we define ``finest'' in
terms of the volume of the torus or the diameter of the fiber. In one
case the answer depends on the structure of $H^1$; in the other case
on the image of $H^1$ under the Ruelle-Sullivan map.  In section
\ref{return} we investigate the implications of the existence of small
cocycles for return equivalence in tiling spaces. Two tiling spaces
are return equivalent if given any two initial transversals (one in
each space) there exist clopen subsets of these transversals so that
the return dynamics of the translation action induced on these clopen
subsets are conjugate. This study is especially well-suited to tiling
spaces for which there exist arbitrarily small cocycles, for then we
can find arbitrarily small clopen subsets of a transversal with
induced $\Z^d$ return actions, and the original space is homeomorphic
to the suspension of these induced actions. In the presence of
arbitrarily small cocycles, this allows us to show that return
equivalence is the same as being homeomorphic for FLC tiling spaces as
indicated in Theorem \ref{returnequiv}.

The organization of this paper is as follows.  In section
\ref{subshifts} we review the definitions of group cohomology and the
Ruelle-Sullivan map in the context of $\Z^d$ actions on Cantor sets,
and review some of the results of \cite{GPSCocycle}. In section
\ref{tilings} we review the connections between minimal $\Z^d$
actions, subshifts, and FLC tiling spaces, and a formulation of tiling
cohomology involving differential forms.  In section \ref{deform} we
show how to implement small shape changes, and prove Theorems
\ref{bigtori} and \ref{morebigtori}, leading to Theorem
\ref{bigsquares}.  Section \ref{finest} concerns the existence  of
``finest'' torus fibrations, and section \ref{return} relates return
equivalence to homeomorphisms. In section \ref{sec-shears}
we investigate the role of shears and discuss Conjecture
\ref{shears}. Finally, in section \ref{DPV} we exhibit
Frank's DPV tiling and show that it has the necessary
cohomological properties as a consequence of its shear
properties. This then completes the proof of Theorem \ref{nogo}.

We thank Natalie Frank, Ray Heitmann,
John Hunton, Henna Koivusalo, Ian Putnam and
Jamie Walton for helpful discussions.  The work of the first author is
partially supported by grant IN-2013-045 from the Leverhulme
Trust. The work of the second author is partially supported by the
National Science Foundation under grant DMS-1101326.

\section{$\Z^d$ actions, cochains, and cohomology}\label{subshifts}

Consider a $\Z^d$ action on a Cantor set $C$. For $n \in \Z^d$, we
denote the action of $n$ on $x$ by $\phi^n(x)$.  A 1-cocycle with
values in $\Z$ is a continuous map $\theta: C \times \Z^d \to \Z$
such that, for all $n,m \in \Z^d$ and all $\chi \in C$,\footnote{We denote
elements of a Cantor set $C$ by Greek letters such as $\chi$,
elements of a subshift $\Xi$ by
Roman letters such as $u$, points in $\Z^d$ by Roman letters such as $n$,
tilings by capital Roman letters such as $T$, and points in $\R^d$ by
Roman letters such as $x$ and $y$.}
\begin{equation} \label{cocycle-condition} \theta(\chi,n+m) = \theta(\chi,n)
+ \theta(\phi^n(\chi),m).
\end{equation}
If $f : C \to \Z$ is a continuous function, then the {\bf
  coboundary} of $f$ is given by
\begin{equation} \label{coboundary} \delta f(\chi,n) =
  f(\phi^n(\chi))-f(x). \end{equation} It's easy to check that all coboundaries
are cocycles.  The first group cohomology of $C$, denoted
$H^1(C)$, is the quotient of the cocycles by the coboundaries.

Given an invariant measure on $C$, we can average a cocycle to get a
function on $\Z^d$:
$$\bar \theta (n) = \int \theta(\chi,n) d\mu(\chi)$$
By the cocycle condition (\ref{cocycle-condition}), this is a linear
function of $n$, hence an element of $Hom(\Z^d, \R) = (\R^d)^*$.  This
averaging procedure is called the Ruelle-Sullivan (RS) map. It is easy
to check that the RS map sends coboundaries to zero, and hence gives a
linear map from cohomology classes to $(\R^d)^*$. (In fact, the RS map
is a ring homomorphism from the full cohomology of the subshift to the
exterior algebra of $\R^d$ \cite{KP}, but here we are only concerned
with the image of $H^1$.)

\begin{ex} If $\theta(\chi,n) = n_i$ (the $i$-th component of $n$) for
  all $\chi$, then $\bar \theta(n)=n_i$, and $\bar \theta$ is the $i$-th
  basis vector in $(\R^d)^*$. This shows that the integer lattice
  within $(\R^d)^*$ is always in the image of the RS map.
 \end{ex}

 In \cite{GPSCocycle}, Giordano, Putnam and Skau studied what they
 call ``small, positive cocycles''. These are cocycles that are
 non-negative for $n$ in a quadrant (say, $n_i \ge 0$ for all $i$),
 and such that $\theta(\chi,n)$ is bounded above and below by
 constants plus small positive multiples of $|n|$ for $n$ in the
 quadrant, with the constants and multiples being independent of
 $\chi$. These cocycles are mapped by RS to small elements in a
 quadrant of $(\R^d)^*$.  The authors discuss consequences of there
 existing arbitrarily small positive cocycles for minimal actions.
 Also, if true, the conjectures would allow one to construct a
Bratteli diagram in which the dynamics are apparent, as one can do for
$\Z$ actions with the Bratteli--Vershik map.
 The main result of the present  paper is to disprove these conjectures.

\section{Subshifts, tilings, and pattern-equivariant
cohomology} \label{tilings}

Let $\aaa$ be a set with $n$ elements, called {\bf letters}. The set
$\aaa$ is called the {\bf alphabet}, and the space $\aaa^{\Z^d}$ is
called the {\bf full shift on $n$ letters}.  If $u \in \aaa^{\Z^d}$
(i.e., $u$ is a map from $\Z^d$ to $\aaa$), then the shift action is
simply
$$ (\phi^n u)(m) = u(m+n).$$
We give $\aaa^{\Z^d}$ the product topology. This is metrizable, and is
often described with a metric in which two functions $u_1, u_2$ are
$\epsilon$-close if they agree exactly on a ball of radius
$1/\epsilon$ around the origin.

A {\bf subshift} $\Xi$ is a subset of $\aaa^{\Z^d}$ that is closed in
the product topology and is invariant under the action of $\phi$. A
subshift is called {\bf aperiodic} if $\phi^n u = u$ implies $n=0$. That is,
if the action of $\phi$ is free. A
subshift is {\bf minimal} if every orbit is dense. A minimal aperiodic
subshift is homeomorphic to a Cantor set since in this case the subshift
has no isolated points, so aperiodic minimal subshifts are special cases of free minimal $\Z^d$ actions on Cantor sets.

Conversely, every expansive $\Z^d$ action on a Cantor set $C$ can be
identified with a subshift. To see this, partition $C$ into finitely many
clopen sets, each of which is smaller than the expansivity
radius. Thus, for any distinct $\chi, \rho \in C$,
there is an $n\in \Z^d$ such that $\phi^n(\chi)$ and $\phi^n(\rho)$ lie in different clopen
sets. Let $\aaa$ be the collection of clopen sets, and for each $\chi \in
C$ define a function $u_\chi: \Z^d \to \aaa$ such that $u_\chi(n)$ is the
clopen set containing $\phi^n(\chi)$. The image of this assignment is a
subshift $\Xi \subset \aaa^{\Z^d}$, and gives an isomorphism between
the given $\Z^d$ action on $C$ and the natural action of $\Z^d$ on
$\Xi$.

The {\bf suspension} of a subshift $\Xi$ is the set $\Xi \times \R^d /
\sim$, where for each $u \in \Xi$, $n \in \Z^d$ and $v \in \R^d$,
$$ (u,n+v) \sim (\phi^n(u), v).$$
There is a natural $\R^d$ action by addition on the second factor.
Such a suspension can be visualized as a tiling space, where the tiles
are unit cubes labeled by the alphabet $\aaa$ and meeting full-face to
full-face. If $u \in \Xi$, then $(u,v)$ is a tiling in which tiles
with labels $u(n)$ occupy $n-v + [0,1]^d$.  (Acting by $v$ on a tiling
means translating the tiling by $-v$, or equivalently translating
the origin, relative to the tiling, by $+v$).

However, there is no reason to restrict attention to tilings by
cubes. We can begin with an arbitrary collection of labeled shapes,
called prototiles, and define tiles to be translates of prototiles.  A
{\bf patch} of a tiling is a finite collection of tiles.  A tiling is
said to have {\bf finite local complexity} (with respect to
translation), or FLC, if for each radius $R$ there are only finitely
many possible patches of diameter less than $R$, up to
translation. This is equivalent to there only being finitely many
prototiles, with only finitely many ways that two tiles can meet.
Suspensions of subshifts are clearly FLC tilings. Conversely, every
FLC tiling space is homeomorphic to the suspension of a subshift
\cite{SW}.  Further, every FLC tiling space is topologically conjugate
to a tiling space in which the tiles are polyhedra that meet full-face to
full-face; we henceforth restrict attention to tilings of this
form. For more details on this construction, and on topological
properties of tiling spaces, see \cite{tilingsbook}.

The metric on a tiling space $\Omega$ is similar to the metric on
subshifts. Two tilings are $\epsilon$-close if they agree on a ball of
radius $1/\epsilon$ around the origin, up to a uniform translation by
up to $\epsilon$. This makes the $\R^d$ action on $\Omega$
continuous. The canonical transversal $\Xi$ of $\Omega$ is the set of
tilings with a vertex at the origin.

\begin{dfn}
  Let $T$ be a tiling. A function $f$ on $\R^d$ is {\bf strongly
    pattern equivariant (with respect to $T$) with radius $r$} if,
  whenever $x,y \in \R^d$ and $T-x$ and $T-y$ agree on a ball of
  radius $r$ around the origin, then $f(x)=f(y)$. In other words, the
  value of $f(x)$ depends only on the pattern of $T$ on a ball of
  radius $r$ around $x$. A function is {\bf strongly
    pattern-equivariant} if it is strongly pattern-equivariant for
  some finite radius $r$. A function is {\bf weakly pattern
    equivariant} if it is the uniform limit of strongly pattern
  equivariant functions. That is, for any $\epsilon>0$ there is an $r$
  such that the value of $f(x)$ is determined, up to $\epsilon$, by
  the pattern of radius $r$ around $x$. When we say a function is
  {\bf pattern equivariant} (abbreviated PE), the pattern equivariance is
  strong unless stated otherwise.
\end{dfn}

A tiling $T$ gives a CW decomposition of $\R^d$ into $d$-cells
(tiles), $(d-1)$-cells (faces), and so on down to 1-cells (edges) and
0-cells (vertices). We can then speak of 0-cochains, 1-cochains, etc.
A $k$-cochain with values in an Abelian group $A$ is an assignment of
an element of $A$ to each oriented $k$-cell in $T$, and hence
by linearity to each $k$-chain in $T$. As usual,
the coboundary of a cochain $\alpha$, applied to a chain $c$, is $\alpha$
applied to the boundary of $c$:
$$ (\delta \alpha)(c) := \alpha(\partial c).$$

As with functions, we say a $k$-cochain is
PE of radius $r$ if its value on a
$k$-cell depends only on the pattern of $T$ on a ball of radius $r$ around the
center-of-mass of the cell. We say a $k$-cochain is (strongly) PE if it is
PE with some finite radius $r$, and we say it is weakly PE if it is the
uniform limit of strongly PE cochains. It is easy to check that the
coboundary of a strongly PE cochain is strongly PE (albeit with a slightly
larger radius), and that the coboundary of a weakly PE cochain is weakly PE.

Consider the cochain complex of (strongly) PE cochains. The cohomology of
this complex is isomorphic to the \v Cech cohomology of the orbit closure of
$T$ with values in $A$ \cite{integer}. In particular,
if the $\R^d$ action on $\Omega$ is minimal, then
the cohomology of this complex is the same for all $T \in \Omega$, and is
isomorphic to the \v Cech cohomology of $\Omega$ with values in $A$.
We can then speak of the PE cohomology of $\Omega$, by which we mean the
cohomology of PE cochains on an arbitrary $T \in \Omega$.

We henceforth restrict our attention to {\em minimal} subshifts and
minimal tiling spaces. That is, {\bf all $\Z^d$ actions on Cantor sets and
  translations actions on tiling spaces are assumed to be minimal
  unless explicitly noted otherwise.}

For $\Z^d$ actions on Cantor sets, there is a natural correspondence
between the first PE cohomology (with values in $\Z$) and the
group-theoretic $H^1$.
If $\alpha$ is a PE cochain, then $\alpha$
extends by continuity to take values on $k$-cells of all tilings in
$\Omega$, since all tilings exhibit the same patterns. Thus we may
speak of the value of a 1-cochain $\alpha$ on an edge of a tiling. If
$\delta \alpha=0$, and if we associate a tiling by unit cubes to each
function in a subshift, then we get a cocycle $\theta(u, n)$ by
applying $\alpha$ to a path from 0 to $n$ in the tiling associated
with $u$. (We call this the {\bf integral} of $\alpha$ along the
path, even though we are merely summing rather than integrating.) Since
$\delta \alpha=0$, this result does not depend on the
path taken.
The cocycle condition (\ref{cocycle-condition}) is
just the statement that the integral from $0$ to $n+m$ equals the integral
from $0$ to $n$ plus the integral from $n$ to $n+m$.

It is often convenient to define tiling cohomology via PE differential
forms.  Indeed, this is the setting in which pattern equivariance was
first defined \cite{Kel, KP}.  A differential form is strongly PE if
all of its coefficients are strongly PE.  It is weakly PE if all of
its coefficients, and the derivatives to all orders of those
coefficients, are uniform limits of strongly PE functions.
As long as the $\R^d$ action on $\Omega$ is minimal, the
cohomology of the de Rham complex of (strongly) PE differential forms
on an arbitrary tiling $T \in \Omega$
is isomorphic to the \v Cech cohomology of $\Omega$ with values in
$\R$ \cite{Kel, KP, integer}.

It sometimes happens that a real-valued strongly PE cochain (or
differential form) $\alpha$ is the coboundary (or exterior derivative)
of a weakly PE cochain (or form). In this case we say that the
(strong) cohomology class of $\alpha$ is {\bf asymptotically
  negligible} \cite{Clark-Sadun}. A theorem of
\cite{Kellendonk-Sadun}, closely related to the classical
Gottschalk-Hedlund theorem, says that the class of a closed PE
1-cochain (or form) $\alpha$ is asymptotically negligible if and only
if the integral of $\alpha$ is bounded.

Recall that $H^1$ of a CW complex is always torsion-free, since the
universal coefficients theorem relates the torsion in $H^1$ to the
(nonexistent) torsion in $H_0$.  Tiling spaces are inverse limits of
CW complexes, so there is never any torsion in $H^1$ of a tiling
space.  This implies that the first integer-valued cohomology of a
(minimal) tiling space can be viewed as a subgroup of the real-valued
cohomology, as represented either by real-valued PE cochains or by PE
differential forms. What characterizes an integral class $[\alpha]$ is
that there exists a radius $r$ such that, for any two occurrences of a
patch $P$ containing a ball of radius $r$, the integral of $\alpha$
from the center of that ball in one occurrence of $P$ to the
corresponding point in the other occurrence is always an integer. (In
the inverse limit construction of a tiling space $\Omega$, such paths
correspond precisely to closed loops in the CW complexes that
approximate $\Omega$.)

The Ruelle-Sullivan map is most easily defined using differential
forms \cite{KP}. If $[\alpha]$ is an integral cohomology class,
represented by the PE form $\alpha$, then we average the value of
$\alpha(0)$ over all tilings in the space, using an invariant
measure. If the space is uniquely ergodic, then this is equivalent to
picking one tiling $T$ and averaging $\alpha(x)$ over larger and
larger balls around the origin.

\begin{dfn}
A closed PE 1-form $\alpha$
is {\bf positive and $\epsilon$-small} if
\begin{itemize}
\item At each point $x$ in each tiling $T$, $|\alpha(x)| < \epsilon$,
\item There is a cone $C \in (\R^d)^*$ such that $\alpha(x)$ applied to any
vector in that cone is everywhere non-negative, and
\item There exists a vector $v$ in that cone such that $\alpha(x)$ applied to
$v$ is everywhere positive.
\end{itemize}
\end{dfn}

We say that a tiling space {\bf has small positive forms} if, for
each $\epsilon>0$, there exist positive $\epsilon$-small forms. This is
the natural analogue, in the setting of differential forms, of the
``small, positive cocycles'' of \cite{GPSCocycle}.

If $\alpha$ is a closed, positive and $\epsilon$-small form, then the
Ruelle-Sullivan map sends $[\alpha]$ to an element of $(\R^d)^*$ of
magnitude less than $\epsilon$. In particular, if $\Omega$ has small positive
forms, then the image of the Ruelle-Sullivan map is not discrete.
Conversely, if $\Omega$ is uniquely
ergodic and the Ruelle-Sullivan image of $[\alpha]$ is nonzero and has
magnitude less than $\epsilon$, then $[\alpha]$ can be represented by
a positive and $\epsilon$-small form \cite{Julien-Sadun}.

\section{Shape changes and virtual eigenvalues}
\label{deform}

The shape of a tile is described by the vectors along all of the edges
around the tile. The assignment of each edge to its corresponding
vector is a (manifestly PE) closed vector-valued cochain, and defines
a class in $H^1(\Omega, \R^d)$. By varying this cochain in a strongly
PE way, we can can obtain a family of tiling spaces, all with the same
combinatorics, but whose tiles have different shapes and sizes. We
call this a ``shape change'' of the original tiling space.
(See \cite{tilingsbook} for more details.)
In \cite{Clark-Sadun}, small shape changes were shown to be parametrized,
up to local equivalence, by $H^1(\Omega, \R^d)$. In
\cite{Kellendonk-Sadun}, all topological conjugacies between tiling
spaces were shown to be a combination of shape changes and local
equivalences (``mutual local derivability'', or MLD). In \cite{Julien-Sadun},
building on \cite{Julien}, all homeomorphisms between uniquely ergodic
tiling spaces were shown to be a combination of shape changes and local equivalences.

Shape changes can also be implemented with differential forms
\cite{Julien-Sadun}. Let $\alpha$
be a closed PE vector-valued 1-form (i.e. an assignment of a square matrix
to each point of a tiling) representing a class in $H^1(\Omega, \R^d)$.
If $\alpha(x)$ is sufficiently close to a
fixed invertible matrix $M$ at each point $x$, we can apply a shape change
in which the displacement from a vertex $x$ to another vertex $y$ becomes
$\int_x^y \alpha$.

Now let $P$ be a patch containing a ball whose radius is greater than
the pattern equivariance radius of $\alpha$. If $x$ and $y$ are
corresponding vertices of different occurrences of $P$ in a tiling, we
say that $y-x$ is a {\bf return vector} of $P$. If all return vectors
are in $\Z^d$, then our tiling space is a fiber bundle over the torus
$\ttt=\R^d/\Z^d$, where the map $\Omega \to \ttt$ just gives the
coordinates of all occurrences of $P$, mod $\Z^d$ (where we have
chosen a particular vertex in $P$ to represent the location of the
patch).  Note that translation in the tiling is equivalent to
translation in the torus, so that this structure is a topological
semi-conjugacy.  Likewise, if all return vectors of $P$ are in a
lattice $L$, we get a bundle over $\ttt_L = \R^d/L$.

If we apply a shape change to $\Omega$, generated by the vector-valued
1-form $\alpha$, then the return vectors $y-x$ are replaced by integrals
$\int_x^y \alpha$. If these are all in $L$, then the shape-changed tiling
space $\Omega'$ is (topologically conjugate to) a bundle over $\ttt_L$.

The following theorem is a slightly stronger restatement of Theorem
\ref{morebigtori}.

\begin{thm}\label{newbigtori} Suppose that $\Omega$ is a uniquely
  ergodic FLC tiling space with free, minimal $\R^d$ action, and that
  $\lambda_1, \ldots, \lambda_d$ are a basis for $(\R^d)^*$. Let $L'$
  be the lattice spanned by the $\lambda_i$'s, dual to a lattice $L
  \subset \R^d$. Then the following are equivalent:
\begin{enumerate}
\item All of the $\lambda_i$'s are in the closure of $RS(H^1(\Omega))$.
\item For any $\epsilon >0$, there is a shape change, implemented by a
  vector-valued 1-form that is pointwise $\epsilon$-close to the
  identity, such that the resulting tiling space $\Omega'$ is
  topologically conjugate to a Cantor bundle over the torus $\R^d/L$.
\item For any $\epsilon > 0$, there is a shape change, implemented by a vector-valued 1-form that is pointwise
$\epsilon$-close to the identity, such that $\lambda_1, \ldots, \lambda_d$ are all topological eigenvalues of $\Omega'$.
\end{enumerate}
\end{thm}

\begin{proof} We will show that (1) implies (2), that (2) implies (3),
  and that (3) implies (1).

  (1) $\Rightarrow$ (2): Let $M_0$ be a matrix whose rows are the
  $\lambda_i$'s. We can find integral classes $[\alpha_1], \ldots,
  [\alpha_d] \in H^1$ such that $M_0^{-1}M$ is $(\epsilon/2)$-close to
  the identity matrix, where $M$ is the matrix whose $i$-th row is the
  image of $[\alpha_i]$ under the Ruelle-Sullivan map.  For
  convenience, package the $d$ scalar-valued cohomology classes
  $[\alpha_1], \ldots, [\alpha_d]$ into a single vector-valued class
  $[\alpha]$, represented by a vector-valued differential form
  $\alpha$.

  Since $\Omega$ is uniquely ergodic, the pointwise ergodic theorem
  implies that $M$ is the spatial average of $\alpha$ over any tiling
  $T$ in $\Omega$, and that convergence to this average is
  uniform. Thus, if $\rho_r$ is a bump function of total integral 1
  and large support (say, achieving a constant maximum value on a ball
  of radius $r \gg 1$ and vanishing outside a ball of radius $r+1$),
  and if we pick $r$ large enough, the convolution of $\alpha$ with
  $r$ will be nearly pointwise constant, and in particular $M_0^{-1}
  (\rho_r * \alpha) (x)$ will be $\epsilon/2$ close to $M_0^{-1}M$,
  and so $\epsilon$-close to the identity.  Let $\tilde \alpha =
  \rho_r * \alpha$. Since $\rho_r$ has compact support, and since
  $\alpha$ is PE, $\tilde \alpha$ is also
  PE, albeit with a radius $r+1$ greater than the pattern
  equivariance radius of $\alpha$. Furthermore, $[\tilde \alpha] =
  [\alpha]$. (For more details of this construction, with precise
estimates on the convergence, see \cite{Julien-Sadun}.)

  Since $[\alpha]$ is an integral class, the integral of $\tilde
  \alpha$ from one occurrence of a (sufficiently large) patch $P$ to
  another occurrence must give an integer. If we then define a
  function $f(x) = \int_0^x \tilde \alpha$, then all occurrences of
  $P$ will have the same value of $f \pmod{\Z^d}$.  Likewise, if we
  define a function $g(x) = \int_0^x M_0^{-1} \tilde \alpha$, then all
  occurrences of $P$ will have the same value of $g \pmod L$. Thus,
  the shape change implemented by $M_0^{-1} \tilde \alpha$ maps
  $\Omega$ to a Cantor bundle over $\R^d/L$.

  (2) $\Rightarrow$ (3): Since $L$ and $L'$ are dual lattices, all
  elements of $L'$ are topological eigenvalues of any bundle over
  $\R^d/L$.

  (3) $\Rightarrow$ (1): If $\lambda_i$ is a topological eigenvalue of
  $\Omega'$, with corresponding eigenfunction $\psi$, then $\frac{-i\,
    d\psi}{2\pi\psi}$ is a constant 1-form (equaling $\lambda_i$) on
  any tiling $T' \in \Omega$ and represents an integral cohomology
  class.  This form pulls back to a nearly-constant (but strongly PE)
  1-form on a tiling $T \in \Omega$, representing the same integral
  class. The spatial average of this 1-form is then $\epsilon$-close
  to $\lambda_i$. Since we can pick $\epsilon$ to be arbitrarily
  small, $\lambda_i$ must be in the closure of the image of the
  Ruelle-Sullivan map.
\end{proof}

To prove Theorem \ref{bigtori}, we simply apply Theorem
\ref{newbigtori} to a basis of $(\R^d)^*$ consisting of $\lambda$ and
all but one of the standard basis vectors for $(\R^d)^*$.

Theorems \ref{bigtori}, \ref{morebigtori}, \ref{bigsquares} and \ref{newbigtori}
all assume unique ergodicity. However, the unique ergodicity was only needed
to produce collections of closed PE 1-forms with desired properties.
The following is the analogue of Theorem \ref{newbigtori} without the
assumption of unique ergodicity.

\begin{thm}\label{notuniquetori} Suppose that $\Omega$ is an
FLC tiling space with free, minimal $\R^d$ action, that $M_0$ is an
invertible matrix, that $\alpha$ is a closed, PE, vector-valued
differential form such that $\|M_0 \alpha(x) - I\| < 1/4$ everywhere,
and that $\alpha$ represents an integer cohomology class.
Let $L$ be a lattice in $\R^d$ that is the integer span of the columns
of $M_0$.
Then there is a shape change, induced by $M_0\alpha$, to a new tiling space
$\Omega'$ that is a Cantor bundle over a torus $\R^d/L$.
\end{thm}

\begin{proof}  Let $f(x) = \int_0^x \alpha$ and let $g(x)=\int_0^x
M_0 \alpha$. The condition $\| M_0 \alpha(x) - I\|<1/4$ is sufficient
to guarantee that $g: \R^d \to \R^d$ is bijective. Since $M_0$ is invertible,
this also implies that $f$ is bijective. As in the proof of
Theorem \ref{newbigtori}, different occurences of a patch $P$ of size
greater than the PE radius of $\alpha$ must have values of $f$ that differ
by elements of $\Z^d$, and so must have values of $g$ that differ by
elements of $L$. Doing a shape change that replaces $x$ by $g(x)$ thus
results in a Cantor bundle over $\R^d/L$.
\end{proof}

If the determinant of $\alpha$ is pointwise small, then the determinant of
$M_0$ must be large, and our torus has
large volume. The less fluctuation there is in $\alpha$, the closer we can
get $M_0\alpha$ to be to the identity (everywhere), and the
closer the shapes and sizes of the tiles in $\Omega'$ will be to the shapes
and sizes in $\Omega$.

\section{Finest bundle}\label{finest}

In \cite{SW} the question is raised whether there is a ``finest''
possible bundle structure of a tiling space. There are a number of
ways in which one might consider one bundle finer than another, and
here we shall give, using two natural notions of measuring fineness,
conditions allowing one to determine in many settings when a
particular bundle structure on a tiling space admits a finest bundle
structure.

In the first notion, we consider the bundle $\Om \overset{p'}{\rightarrow}
\Torus^d$ {\bf finer} than the bundle $\Om\overset{p}{\rightarrow}
\Torus^d$ if there exists a commutative diagram
\begin{equation}\label{finer}
\xymatrix{
\Om\ar[d]_{p}\ar[rd]^{p'}& \\
\Torus^d&\Torus^d\ar[l]^{\pi}}
\end{equation}
in which the mapping $\pi$ is a covering map, and we denote this
relation as $p \preceq p'.$ This is a natural notion in the bundle
category since the bundles $\Torus^d \to \Torus^d$ are covering
maps. If the degree of $\pi$ is greater than one, we consider $p'$ to
be {\bf strictly finer} than $p$ and denote this by $p \prec p'.$ We
call an abelian group $G$ {\bf infinitely generated but of rational
  rank $d$} if $G$ is not finitely generated but $G\otimes\Q$ is
isomorphic to $\Q^d.$

\begin{thm} Let $\Om$ be a minimal FLC tiling space of dimension $d.$
  Suppose further that $H^1(\Om,\Z)$ contains no subgroup that is
  infinitely generated but of rational rank $d.$ Then given any
  bundle projection $\Om \overset{p}{\rightarrow} \Torus^d,$ there is
  a bundle projection $\Om\overset{p'}{\rightarrow} \Torus^d$
  satisfying: $p \preceq p'$ and there is no bundle projection $p''$
  with $p'\prec p''.$ In general there is no such maximal element.
\end{thm}

\begin{proof}
  Assume that we have an infinite sequence $p_n: \Om \to \Torus^d$ of
  bundle projections satisfying $p_n \prec p_{n+1}$ for all $n.$
  Consider then the following commutative diagram.

\begin{equation}\label{eq-collapse}
\xymatrix{
 & \Om\ar[ld]_{p_1}\ar[d]^{p_2}
\ar[rd]^{p_n}& \\
\Torus^d\;\; & \;\;\;\Torus^d\ar[l]^{\pi_1} \cdots &    \ar[l]^{\pi_n}\Torus^d  \cdots
}
\end{equation}
If we consider the tiling space $\Om$ to have the translational $\R^d$
action given by the suspension construction based on the bundle
projection $p_1$ (which sometimes is and sometimes is not conjugate to
the original action, see \cite{Clark-Sadun}), then each $p_n$
semi--conjugates the translational action on $\Om$ with the natural
translation action of $\R^d$ on $\Torus^d$, and the induced map to the
inverse limit $\Om\overset{p}\to \lim_{\leftarrow} \{\Torus^d,
\pi_n\}$ semi--conjugates the translation action on $\Om$ with the
$\R^d$ action on $\lim_{\leftarrow} \{\Torus^d, \pi_n\}$ by the
commutativity of the diagram. Since the action on $\lim_{\leftarrow}
\{\Torus^d, \pi_n\}$ is equicontinuous, $p$ induces an injection
$p^*:H^1(\lim_{\leftarrow} \{\Torus^d,\pi_n\},\Z)\to H^1(\Om,\Z),$
see, e.g., \cite{BKS}. As $H^1(\lim_{\leftarrow} \{\Torus^d,
\pi_n\},\Z)$ is infinitely generated but of rational rank $d$,
$H^1(\Om,\Z)$ contains a subgroup that is infinitely generated but of
rational rank $d.$
\end{proof}

If $H^1(\Om,\Z)$ contains a subgroup that is infinitely generated
but of rational rank $d,$ there is in general no such maximal
element. For example, in substitution tiling spaces arising from
substitutions of constant length, or in the chair tilings of the plane, we
will have exactly such a sequence of increasingly finer projections.

The theorem does show the existence of a finest bundle
projection, for example, for the Penrose tiling space or any Euclidean
cut and project tiling space.

There is a second notion of finer for which the notions we
have developed here have direct implications. We say that the bundle
$\Om \overset{p'}{\rightarrow} \Torus^d$ is \emph{fiber finer} than
the bundle $\Om \overset{p}{\rightarrow} \Torus^d$ if
$p'^{-1}(\mathbf{0}) \subsetneqq p^{-1}(\mathbf{0}).$ The following is
then a corollary to Theorem~\ref{bigsquares}.

\begin{thm}
  Let $\Omega$ be a uniquely ergodic FLC tiling space with a minimal
  $\R^d$ action. If the image of $H^1(\Omega)$ is dense under the RS
  map, then $\Omega$ admits a sequence of bundle projections $p_n$
  with $p_{n+1}$ fiber finer than $p_n$ for all $n.$
\end{thm}

\begin{proof} Consider the sequence of bundle projections $p_n$ 
corresponding to $\R^d/(n! \Z)^d$ as in Theorem~\ref{bigsquares}. 
 \end{proof}

But by the above, the projections $p_n$ will not generally be related
by the relation $\prec$. One could also consider this result in terms of induced actions. A
tiling space with a given bundle structure
$\Om\overset{p}{\rightarrow} \Torus^d$ is also the suspension of a
$\Z^d$ action $\phi$ given by the global holonomy to a fiber $F$ of
the projection as discussed in section~\ref{tilings}. When the
conditions of the theorem are met, for any given clopen set $K$ of $F$
it follows that one can find an induced $\Z^d$ action $\phi_{K'}$ on
some proper clopen $K'\subset K.$ This induced action is the analogue
of the first return map for a flow.

\section{Return equivalence in tiling spaces}\label{return}

In \cite{CHL2013a} the authors developed the notion of return
equivalence for matchbox manifolds, a class of foliated spaces
including FLC tiling spaces which are generally characterized by
having totally disconnected transversals. Two matchbox manifolds
$\cM_1,\cM_2$ are {\bf return equivalent} if given any clopen subsets
$K_i$ of transversals of $\cM_i$, there exist clopen subsets
$V_i\subset K_i$ such that the induced holonomy pseudogroups on the
$V_i$ are isomorphic. Roughly, a pseudogroup on $X$ is a non--empty
collection of homeomorphisms between open subsets of $X$ that is
closed under inverses, compositions, restrictions and unions. An
isomorphism between a pseudogroup $\Psi_X$ on $X$ and a pseudogroup
$\Psi_Y$ on $Y$ is given by a homeomorphism $h:X \to Y$ such that for
each $g\in \Psi_X,$ $h\circ g \circ h^{-1}\in \Psi_Y$ and conversely,
and this correspondence respects composition. In the case of an FLC
tiling space $\Om$, the induced holonomy pseudogroup on a transversal
$\T$ is easier to conceive than is usual since it is generated by all
the return maps of the $\R^d$ translation action on $\Om$ to $\T.$ (We
refer the reader to \cite{CHL2013a} for detailed definitions of these
concepts in the general case.) In the case of the suspension of a
$\Z^d$ action $\phi$ on a Cantor set $C$ as in section~\ref{tilings},
$C$ can be regarded as a transversal, and the induced holonomy
pseudogroup on $C$ is the pseudogroup on $C$ generated by
homeomorphisms generating $\phi.$

In \cite{CHL2013a} it is shown that if $\cM_1,\cM_2$ are homeomorphic
{\bf minimal} matchbox manifolds (meaning that each leaf is dense),
then $\cM_1$ and $\cM_2$ are return equivalent. To see why this should
be true in the context of minimal, aperiodic FLC tiling spaces, first
consider that a homeomorphism $h:\Om_1 \to \Om_2$ maps local
transversals (which we regard as subspaces of the domain of a chart)
of $\Om_1$ to those of $\Om_2$ and conversely, and similar facts hold
for atlases of charts. Thus, it suffices to show that given any two
clopen subsets $U_i$ of local transversals $\T_i$ of a single space
$\Om$ there are clopen subsets $V_i\subset U_i$ with isomorphic
induced holonomy pseudogroups. By aperiodicity and minimality of the
translation action $\Phi: \R^d\times \Om \to \Om,$ there is a
topological disk $D \subset \R^d$ such that for some clopen sets
$V_i\subset U_i$ we have for $B:=\Phi(D\times V_1)$ that $B \,\cap\,
V_i= V_i$ and, for each $v\in V_1$, $\Phi(D\times \{v\})$ intersects $V_2$ in a single point. This yields for each point of $V_1$ a uniquely paired
element of $D$ which $\Phi$ maps to a uniquely determined element of
$V_2.$ This correspondence between the elements of $V_1$ and $V_2$ is
the homeomorphism that yields the isomorphism between the respective
holonomy pseudogroups. Similar considerations show that in the case of
the suspension of a $\Z^d$ action $\phi$ on a Cantor set $C,$ any
sufficiently small transversal has holonomy pseudogroup isomorphic to
the holonomy pseudogroup of a clopen subset of $C.$

It is also shown in \cite{CHL2013a} that there are large classes of
matchbox manifolds for which return equivalence implies homeomorphism.
A key tool in most of the proofs is the basic fact that if
$\cM_1,\cM_2$ are both bundles over the same closed manifold $B$ with
conjugate global holonomy actions on the fiber, then $\cM_1$ and
$\cM_2$ are homeomorphic, see, e.g., \cite[Theorem
2.1.7]{CandelConlon2000}. It is however also determined in
\cite{CHL2013a} that there are return equivalent minimal matchbox
manifolds that both have the structure of the total space of a
principal bundle over a surface of genus 2 which are return equivalent
but not homeomorphic.

The question of whether ``return equivalent'' implies ``homeomorphic'' for
tiling spaces involves both topological and ergodic assumptions:

\begin{thm} \label{returnequiv} Let $\Om_1$ be a uniquely ergodic FLC
  tiling space with a free, minimal $\R^d$ action for which the image of
  $H^1(\Om)$ is dense. Then an aperiodic minimal FLC tiling space $\Om_2$ is
  return equivalent to $\Om_1$ if and only if $\Om_1$ and $\Om_2$ are
  homeomorphic.
\end{thm}

\begin{proof}
  We regard the $\Om_i$ as suspensions of minimal, aperiodic $\Z^d$
  actions $\phi_i$ on Cantor sets $C_i.$  Assuming that the $\Om_i$ are return equivalent,
  we must show that they are homeomorphic. Let $\T_i$ be given
  transversals of the $\Om_i.$ As discussed above, we may consider
  without loss of generality the $\T_i$ to be clopen subsets of $C_i$
  with holonomy actions that induced by $\phi_i.$ By return
  equivalence, there exist clopen subsets $K_i \subset \T_i \subset
  C_i$ and a homeomorphism $h:K_1 \to K_2$ that conjugates the
  pseudogroups on the $K_i$ induced by $\phi_i.$ By our hypothesis on
  $\Om_1$ and Theorem \ref{bigsquares} we then know that $K_1$ has a
  clopen subset $V_1$ such that $\phi_1$ induces a $\Z^d$ action
  $\phi_1|V_1$ of $V_1$ and $\Om_1$ is the suspension of this
  restricted action. The induced pseudogroup action on $V_1$ is then
  the pseudogroup generated by $\phi_1|V_1$ and $h$ conjugates this
  $\Z^d$ action with a $\Z^d$ action $\phi_2|V_2$ on $V_2:=h(V_1)$ which in
  turn is induced by the holonomy action of $\phi_2$ on $V_2.$ Then
  $\Om_2$ is homeomorphic to the suspension of $\phi_2|V_2$ and so
  $\Om_1$ and $\Om_2$ are homeomorphic as they are both suspensions
  over the torus of conjugate actions on a Cantor set.
\end{proof}

The key ingredient in this proof is the existence of arbitrarily
small, positive forms. We suspect that a similar result holds under
much weaker conditions, but the proof would necessarily involve
different ideas due to the lack of bundles over arbitrarily large tori
without a small form condition.

\section{The role of shears}\label{sec-shears}

In preparation for a discussion of Conjecture \ref{shears}, we show that the
topological eigenvalues of a subshift with shears are restricted.

\begin{thm} \label{baby-shears} Let $\Xi$ be a minimal and uniquely
  ergodic $\Z^2$ subshift, and let $\Omega$ be the suspension of
  $\Xi$.  Let $\lambda = (\lambda_x,\lambda_y)$ be a topological
  eigenvalue of $\Omega$.
\begin{enumerate}
\item If $\Xi$ admits a horizontal shear, then $\lambda_x \in \Z$.
\item If $\Xi$ admits a vertical shear, then $\lambda_y \in \Z$.
\item If $\Xi$ admits both a horizontal shear and a vertical shear, then $\lambda \in \Z^2$.
\end{enumerate}
\end{thm}

\begin{proof} Suppose that $\lambda$ is a topological eigenvalue with
  eigenfunction $\psi$, normalized so that $|\psi|=1$
  everywhere. Since $\psi$ is a continuous function on a compact space
  $\Omega$, it is uniformly continuous. So for each $\epsilon >0$
  there exists a $\delta>0$ such that if two tilings $T, T'$ are
  $\delta$-close, then $|\psi(T) - \psi(T')| < \epsilon$.

  Now suppose that $\Xi$ admits a horizontal shear with function
  $u$ and corresponding tiling $T_1$. Let $T_2$ be a tiling in
  $\Omega$ that agrees with $T_1$ on $\R \times [N, \infty)$ and
  agrees with $T_1-(1,0)$ on $\R \times (-\infty, N']$.  Pick
  $\epsilon>0$ and a constant $K > \delta^{-1} + \max(|N|,|N')$.

  Since $T_1-(0,K)$ and $T_2-(0,K)$ agree on $\R \times [N-K,
  \infty)$, they are $\delta$-close, and so $\psi(T_1-(0,K))$ and
  $\psi(T_2-(0,K))$ are $\epsilon$-close.  By the eigenvalue property,
  $\psi(T_2-(0,-K)) = \exp(-4\pi i \lambda_y K) \psi(T_2-(0,K)).$
  Since $T_2-(0,-K)$ and $T_1-(1,-K)$ agree on $\R \times (\-\infty,
  K+N')$, they are $\delta$-close, so $\psi(T_2-(0,-K))$ and
  $\psi(T_1-(1,-K))$ are $\epsilon$-close.  Finally,
  $\psi(T_1-(1,K))$ equals $\exp(4 \pi i \lambda_y K)
  \psi(T_1-(1,-K))$ by the eigenvalue property. That is,
\begin{eqnarray*}
  \psi( T_1-(1,K)) & = & \exp(4 \pi i \lambda_y K) \psi(T_1-(1,-K)) \cr
  & \approx & \exp(4 \pi i \lambda_y K) \psi(T_2-(0,-K)) \cr
  & = & \psi(T_2-(0,K)) \cr
  & \approx & \psi (T_1-(0,K))
\end{eqnarray*}
However, $\psi(T_1-(1,K)) = \exp (2 \pi i \lambda_x)
\psi(T_1-(0,K))$, so $\exp(2 \pi i \lambda_x)$ must be $2\epsilon$-close to 1.
Since $\epsilon$ was arbitrary, $\exp(2 \pi i \lambda_x)$
must be exactly 1, and $\lambda_x$ must be an integer.

This proves statement (1). Statement (2) is similar, and statement (3)
follows from statements (1) and (2).
  \end{proof}

It is tempting to try to prove Conjecture \ref{shears} by
modifying the proof of Theorem \ref{baby-shears} to take into
account the effects of the small shape change. Unfortunately, we have not
succeeded in this approach, insofar as small shape changes, integrated over
large distances, can result in non-negligible displacements.

Instead, we study specific examples where we can track the effects of the shears
directly.
In \cite{FrankSadunCoho}, the authors study tiling spaces with continuous
shears (and with infinite local complexity), and show that the shears serve to kill off parts of $H^1$.
In the following section, we obtain qualitatively similar results for tilings
with FLC.

\section{The counterexample} \label{DPV}

Natalie Frank's Direct Product Variation (DPV) tiling (see
\cite{Frank-primer, FR} for basic properties and \cite{FrankSadunCoho,
  FrankSadunTP} for further development) comes from a fusion rule
\cite{FrankSadunGD}, which can be considered as a generalization of a
substitution, yielding a hierarchical structure. In the particular
example $\Om_F$ we consider, there are four tiles, $a$, $b$, $c$, and
$d$, which we take to be unit squares in $\R^2.$ All tilings in
$\Om_F$ consist of translates of these four tiles.  For the fusion
rule used to construct $\Om_F,$ supertiles $P_n(a, b, c, \hbox{ or }
d)$ are defined recursively. The 0-supertiles are just the tiles
themselves. For $n \ge 0$, we combine $n$-supertiles into
$n+1$-supertiles $P_{n+1}(a,b,c,d)$ by the following combinatorial
rule:

\begin{equation} \label{DPV-fusion}
 P_{n+1}(a) = \left [ \begin{matrix} P_n(b)&P_n(d)&P_n(d)&P_n(d) \cr P_n(c)&P_n(c)&P_n(a)&P_n(d) \cr P_n(d)&P_n(d)&P_n(b)&P_n(d) \cr P_n(d)&P_n(d)&P_n(b)&P_n(c) \end{matrix} \right ], \quad
P_{n+1}(c) = \left [ \begin{matrix} P_n(b) \cr P_n(b) \cr P_n( b) \cr
    P_n(a) \end{matrix} \right ], $$ $$ P_{n+1}(b) = \left
  [ \begin{matrix} P_n(a)&P_n(c)&P_n(c)&P_n(c) \end{matrix}\right ],
\quad P_{n+1}(d) =[ P_n(a)].
\end{equation}

\hspace{.05in} \includegraphics[width=0.4\textwidth]{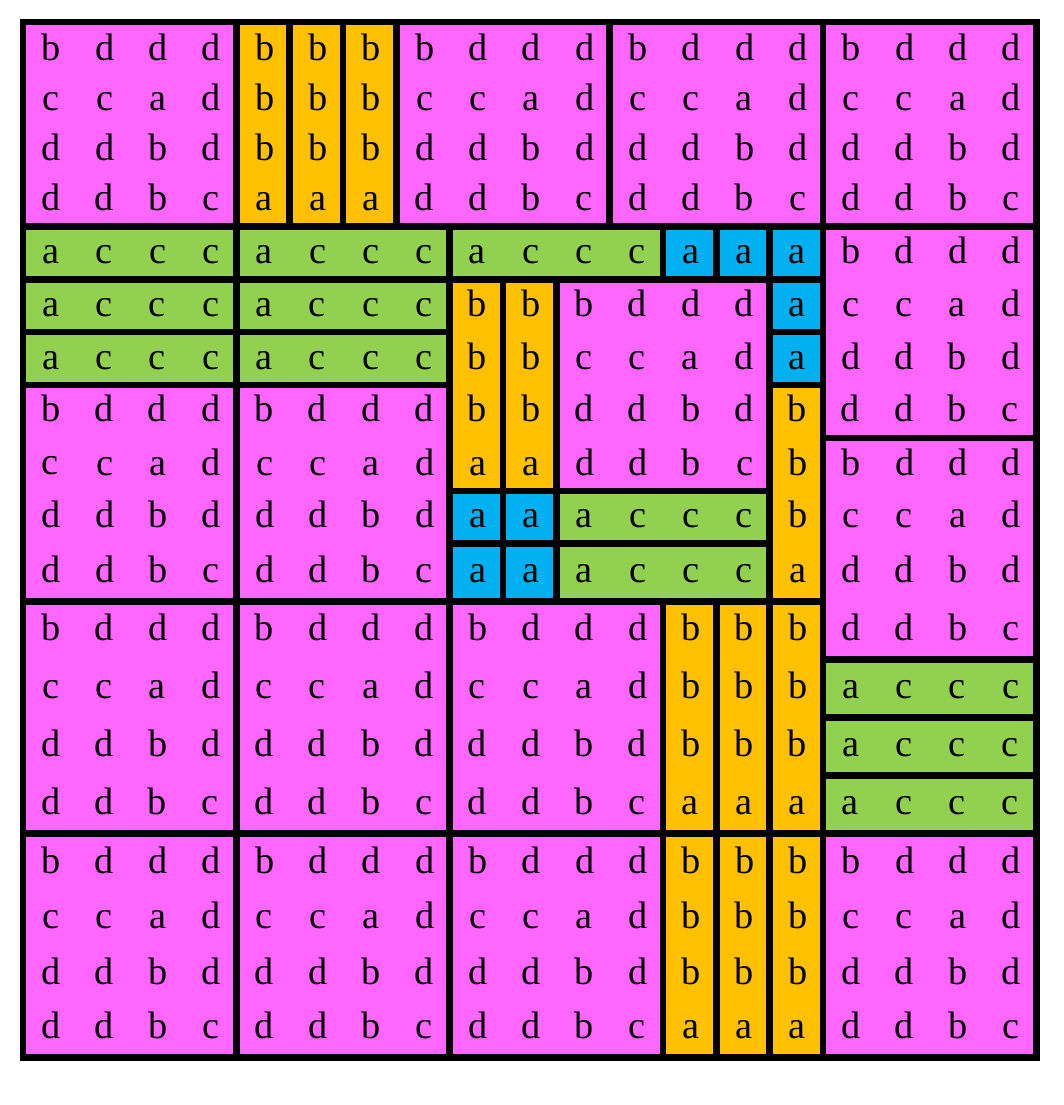} \hspace{.1in} \includegraphics[width=0.421\textwidth]{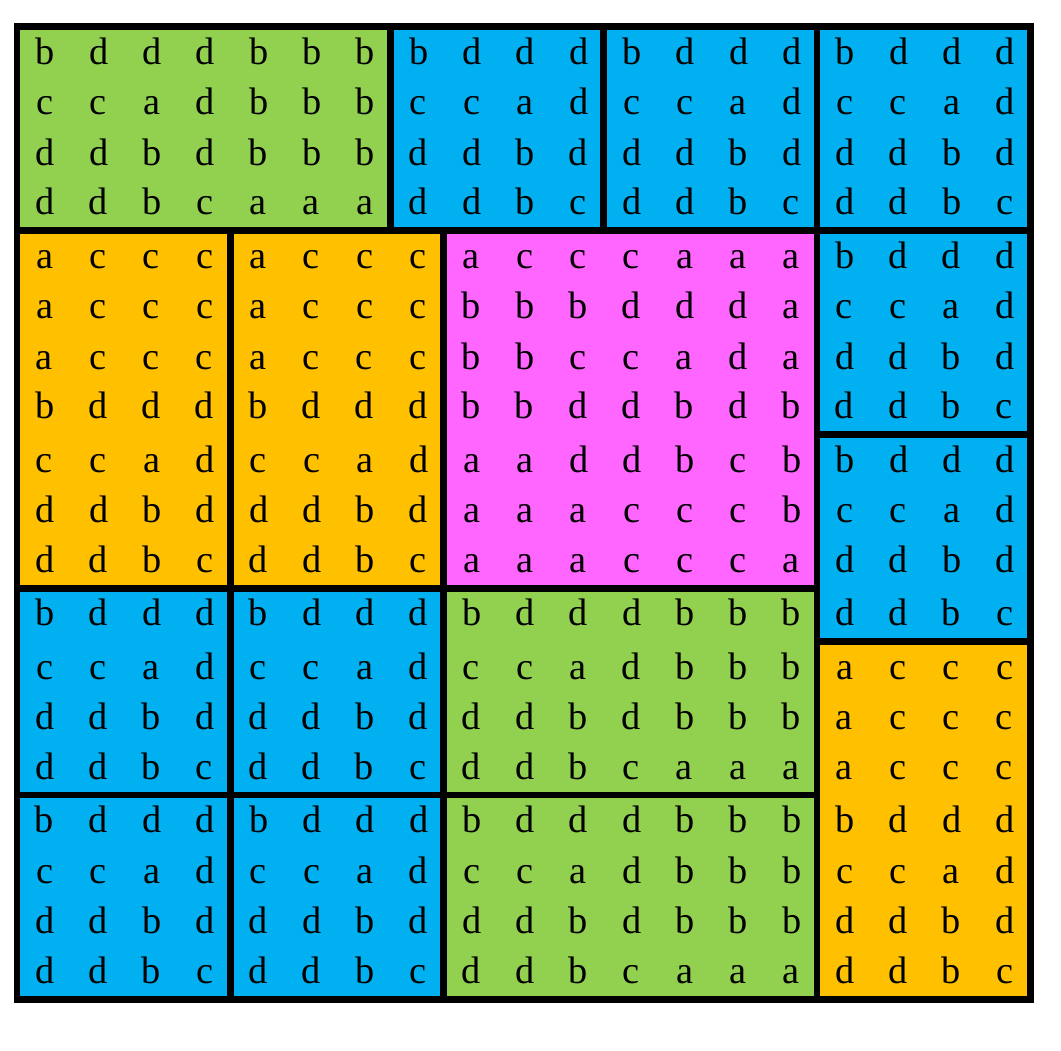}

\begin{center}
{\small The 1 and 2-supertiles within $P_3(a)$}
\end{center}

\medskip Notice that the fusion rule (\ref{DPV-fusion}) admits an
involution in which we swap the vertical and horizontal directions,
while swapping the labels $b$ and $c$ at the same time. Thus, many of
the facts established for the horizontal direction apply equally to
the vertical direction.

By the primitivity of the fusion rule, the tilings in $\Om_F$ consist
of those tilings $T$ for which every patch of $T$ is contained in
$P_n(a)$ for some $n\geq 0$, where we can replace $a$ with any of the
other 3 symbols. Using the standard equivalence between the occurrence
of patches with bounded gaps and the minimality of the translation
action, this allows one to see the minimality of the translation
action directly. Any patch $P$ of a tiling in $\Om_F$ occurs in
$P_n(a)$ for some $n$ by construction. Now every $(n+1)$-supertile
contains at least one copy of $P_n(a)$, and the $(n+1)$-supertiles
appear with bounded gaps because they have bounded
diameters.\footnote{We are using the fact that every tiling $T\in
  \Om_F$ admits a decomposition into $(n+1)$-supertiles. Although the
  uniqueness of such a decomposition is sometimes subtle, the
  existence follows from the axioms of hierarchical tilings
  \cite{FrankSadunGD}.}  Similar arguments as used for showing the
freeness of the translation action for substitution tiling spaces can
be applied to show that the translation action on $\Om_F$ is free, but
this can also be directly inferred directly from the existence of the
shears established in Theorem~\ref{DPVshears}.  That the translation
action on $\Om_F$ is uniquely ergodic follows from the general result
\cite[Cor. 3.10]{FrankSadunGD} for fusion tilings.

To see how the shears come about, consider the way that the supertiles
are situated within supertiles of larger order. For $n>0$, the different
$n$-supertiles have different sizes, but the fusion is set up so that
things always fit together. In fact, $P_n(a)$ and $P_n(d)$ are always
squares, with the $a$-supertile being roughly $(1+\sqrt{13})/2$ times
wider than the $d$ supertile. $P_n(b)$ and $P_n(c)$ are wide and tall
rectangles, respectively, with aspect ratios approaching
$1:(1+\sqrt{13})/2$ as $n \to \infty$.

However, as seen in the above figure,
the $n$-supertiles within an $(n+1)$-supertile do {\em not}
typically meet full-edge to full-edge. They meet with a number of
possible offsets. The $(n-1)$-supertiles on the boundaries of these
$n$-supertiles meet with a greater number of possible offsets.  The
$(n-2)$-supertiles on the boundaries of {\em these} $(n-1)$-supertiles
meet with even more possible offsets.

\begin{thm}\label{DPVshears}
The Frank DPV tiling $\Om_F$ admits shears in both directions.
\end{thm}

\begin{proof}  Most of this proof is already well known, in particular
the fact that it is possible to apply infinitely many shears to a legal
tiling and obtain another legal tiling \cite{FR}. What is new is proving
that shears by {\em arbitrary} integer distances are possible.
We will show this for shears in the horizontal
direction. The vertical direction is similar by the symmetry in the fusion rule.

The horizontal shears come about from a mismatch in the 1-dimensional
dynamics just above and below a horizontal boundary between high-order
supertiles. On the north side of the boundary, we have a 1-dimensional
substitution $\sigma_{\mathbf{n}}$ that is obtained by looking at the
bottom row of each supertile.  That is,
$$ \sigma_{\mathbf{n}}(a)=ddbc, \quad \sigma_{\mathbf{n}}(b)=accc, \quad \sigma_{\mathbf{n}}(c)=a, \quad \sigma_{\mathbf{n}}(d)=a.$$
Likewise, on the south side of the boundary, we have a substitution
$\sigma_{\mathbf{s}}$ derived from the top row of the 2-dimensional
supertiles:
$$\sigma_{\mathbf{s}}(a)= bddd, \quad \sigma_{\mathbf{s}}(b)=accc, \quad \sigma_{\mathbf{s}}(c)=b, \quad \sigma_{\mathbf{s}}(d)=a.$$

On both sides of this ``fault line'', the $n$-supertiles come in two
widths, namely the entries of $(1,1) M^n$, where $M= \left
  ( \begin{smallmatrix} 1&1 \cr 3&0 \end{smallmatrix} \right )$.
These widths are relatively prime. To see this, note that the matrix
$M$ has determinant $-3$, and so is invertible over all $\Z_p$ with
$p$ a prime other than 3. Since the widths of the basic tiles are not
zero (mod $p$), we cannot have both widths of the $n$-supertiles
divisible by $p$. Meanwhile, it is easy to check that $(1,1) M^n =
(1,1)$ (mod 3), and hence that no supertiles have lengths divisible by
3. Thus the widths of the $n$-supertiles share no common factors.

Comparing the 1-dimensional substitution dynamics above and below the
line, we find regions where there are more wide supertiles above and
narrow supertiles below, or vice-versa. This difference in population
above and below the line (sometimes called the {\em discrepancy}) is known to
be unbounded \cite{FR}. This means that the wider $n$-supertiles
above and below the fault line are offset by arbitrary multiples of
the narrower width (mod the wider width). Since the two widths are
relatively prime, arbitrary offsets between wide $n$-supertiles are
possible. Since the substitution $\sigma_{\mathbf{n}}$ is primitive,
this implies that arbitrary offsets between {\em any} supertile above,
and {\em any} supertile below, are possible.

Now consider a tiling $T$ where the upper half plane consists of an
infinite-order supertile (meaning the union of supertiles of higher
and higher order), and the lower half plane consists of another
infinite-order supertile. That it is possible to form such a tiling
follows from an application of the Extension Theorem \cite[3.8]{GS}
where we use patches formed by two supertiles of increasing order
which share a boundary along the $x$--axis.  Translating the lower
half plane in $T$ sideways by an arbitrary integer yields another
legal tiling $T'\in\Om_F$, since all the patterns of $T'$ near the
$x$-axis consist of northern supertiles meeting southern supertiles
with arbitrary offsets, and these in turn are already found elsewhere
in the tiling, from which it follows that all patterns in the
resulting tiling are legal.
\end{proof}

Since the tiling just above the fault line is governed by
$\sigma_{\mathbf{n}}$, it is useful to establish the topology of the
1-dimensional tiling space $\Omega_{\sigma_{\mathbf{n}}}$ generated by
the substitution $\sigma_{\mathbf{n}}$.

\begin{lem} $H^1(\Omega_{\sigma_{\mathbf{n}}}; \R) = \R^4$, and the group $H^1_{an}(\Omega_{\sigma_{\mathbf{n}}}; \R)$ of asymptotically
negligible classes is trivial.
\end{lem}

\begin{proof} We compute $H^1(\Omega_{\sigma_{\mathbf{n}}}, \R)$ using
  the methods of Barge and Diamond \cite{BD}.  Let $A = \left
    ( \begin{smallmatrix} 0&1&1&1 \cr 1&0&0&0 \cr 1&0&0&0 \cr 2&3&0&
      0 \end{smallmatrix} \right )$ be the substitution matrix of
  $\sigma_{\mathbf{n}}$.  Let $X$ be the graph that describes all
  possible adjacencies between final letters of $n$-supertiles and
  beginning letters of the subsequent $n$-supertile. For $n \ge 3$,
  there are four possible adjacencies: $a.a$, $a.d$, $c.a$ and $c.d$,
  so $X$ is isomorphic to a circle. According to \cite{BD}, the
  cohomology of $\Omega_{\sigma_{\mathbf{n}}}$ fits into the exact
  sequence
$$ 0 \to \tilde H^0(X; \R) \to \lim (\R^2, A^T) \to H^1(\Omega_{\sigma_{\mathbf{n}}}; \R) \to H^1(X; \R) \to 0.$$
Since $\tilde H^0(X; \R)=0$ and $H^1(X; \R)=\R$, and since $A$ has
rank 3, with eigenvalues $(1 \pm \sqrt{13})/2$, $-1$ and $0$,
$H^1(\Omega_{\sigma_{\mathbf{n}}}; \R) = \R^4$.  The substitution
homeomorphism maps $\Omega_{\sigma_{\mathbf{n}}}$ to itself, and so
induces a linear transformation on $H^1(\Omega_{\sigma_{\mathbf{n}}};
\R)$.  The eigenvalues of this linear transformation are $(1 \pm
\sqrt{13})/2$, $-1$ and $1$, with the 1 coming from the action of
substitution on $H^1(X; \R)$.

By a theorem of \cite{Clark-Sadun}, $H^1_{an}$ is the contracting
subspace of $H^1$ under substitution. Since all eigenvalues are of
magnitude 1 or greater, $H^1_{an}$ is trivial.
\end{proof}

\begin{thm}\label{mainthm}
  For the Frank DPV tiling, the image of $H^1(\Omega_F)$ under the
  Ruelle-Sullivan map is $\Z^2$. \end{thm}

\begin{proof} Suppose that $\alpha$ is a PE differential form
  representing an integral class in $H^1(\Omega)$, and that the
  Ruelle-Sullivan map sends $[\alpha]$ to $(\mu, \nu) \in \R^2$.  The
  proof can be broken down into 5 steps.
\begin{enumerate}
\item Using the fact that $\Omega$ admits horizontal shears, we bound
  the fluctuations in $\int \alpha$ along different intervals of a
  fixed horizontal line in a fixed class of tilings.
\item We use properties of $\Omega_{\sigma_{\mathbf{n}}}$ to control
  the values of $\alpha$ along this horizontal line. In particular, we
  show that $\alpha = \mu\, dx + d \beta$ along this line, where $\beta$
  is a strongly PE function.
\item Using the fact that $[\alpha]$ is an integral class, we show
  that for all sufficiently large values of $n$, $\mu$ times the
  length of any $n$-supertile must be an integer.
\item Since the lengths of the $n$-supertiles are relatively prime,
  $\mu$ must itself be an integer.
\item The same arguments, with the roles of $x$ and $y$ reversed, show
  that $\nu$ is an integer.
\end{enumerate}

{\bf Step 1:} Let $r$ be the PE radius of $\alpha$.  Pick an integer
$N>r$ and another integer $n$ such that the height of the smallest
$n$-supertile is greater than $N+r$.  Consider tilings where there is
a boundary between infinite-order supertiles on the $x$ axis and where
the vertices are located on $\Z^2$.  Let $\ell_+$ and $\ell_-$ be the
horizontal lines $y=N$ and $y=-N$.  For any tiling in our class, we
consider the integrals $\int_{(0,N)}^{(L,N)} \alpha$ along $\ell_+$
and $\int_{(0,-N)}^{(L,-N)} \alpha$ along $\ell_-$, where $L$ is an
arbitrary positive integer. Since $\alpha$ is closed, the difference
between these integrals is the difference between the integrals along
vertical paths from $(0,N)$ to $(0,-N)$ and from $(L,N)$ to $(L,-N)$.
However, these vertical integrals are of bounded length, independent
of $L$ and independent of the tiling in question. Since $\alpha$ is
PE, and hence bounded, there is a constant $K$ such that
$\int_{(0,N)}^{(L,N)} \alpha$ is always within $K$ of
$\int_{(0,-N)}^{(L,-N)} \alpha$, regardless of the tiling in question
or the length $L$.

Since $N>r$, the integrals along $\ell_+$ and $\ell_-$ depend only on
the tiling structure in the upper and lower half-plane, respectively.
Furthermore, since the tiling space admits shears, a given horizontal
path along $\ell_+$ can be paired with any horizontal path of the same
length along $\ell_-$.  Thus the integrals along all such paths of
length $L$ along $\ell_-$ must take values within $2K$ of one another,
and likewise the integrals along all paths of length $L$ in
$\ell_+$. In particular, the integral of $\alpha$ along any path of
length $L$ in $\ell_+$ must be within $2K$ of $L\, \cdot$ (average
value of $\alpha$ on $\ell_+$).

{\bf Step 2:} Since $\alpha$ is closed, all horizontal lines in a
given tiling $T$ give the same average value $\lim_{L \to \infty}
\frac{1}{L} \int_p^{p+(L,0)} \alpha$, and the unique ergodicity of the
translation action implies that (the horizontal component of)
$RS([\alpha])$ can be computed by averaging (the horizontal component
of) $\alpha(x)$ over an arbitrary tiling $T$, and is equal to this
common linear average value.  Thus the average value of $\alpha$ on
$\ell_+$ is exactly $\mu$.  We can then write $\alpha$, restricted to
$\ell_+$, as $\mu dx + \alpha_0$, where $\alpha_0$ has average zero.
The results of the previous paragraph imply that the integral of
$\alpha_0$ along any path in $\ell_+$ is bounded by $2K$.

Note that $\ell_+$ is a row ($N$ from the bottom) of a
sequence of $n$-supertiles at the bottom of an infinite-order
supertile.
We associate to $T$ a tiling $T_0 \in \Omega_{\sigma_{\mathbf{n}}}$
given by the
sequence of $n$-supertiles lying just above the $x$-axis.  $T_0$ is a
1-dimensional tiling whose tiles have the labels $a$, $b$, $c$, and
$d$ and the widths of the corresponding $n$-supertiles in
$\Omega_F$.  On $T_0$ we define a 1-cochain $\tilde \alpha_0$ whose value
on a tile (say, running from $x_1$ to $x_2$)
is the integral of $\alpha_0$ across the corresponding
stretch of $\ell_+$ (i.e., from $(x_1,N)$ to $(x_2,N)$).
Since $\ell_+$ lies a distance $r$ or greater from the top or bottom
of these supertiles, this integral depends only on which of the four
supertile types we are working and on the identities of the
supertile's predecessors
and successors to distance $r$. In particular, $\tilde \alpha_0$ is
strongly PE.

Since the integral of $\alpha_0$ along $\ell_+$ is bounded, the integral
of $\tilde \alpha_0$ is bounded, so $\tilde \alpha_0$ must represent an
asymptotically negligible class in $H^1(\Omega_{\sigma_{\mathbf{n}}})$.
However, $H^1_{an}(\Omega_{\sigma_{\mathbf{n}}})$ is trivial.
Thus $\tilde \alpha_0$ represents the zero class, and is the derivative of
a PE function $\tilde \beta$ on $T_0$.
That is, if $P_0$ is a patch of sufficient
length in
$\Omega_{\sigma_{\mathbf{n}}}$ (specifically, greater than twice the PE radius of
$\tilde \beta$), and if this patch occurs at two
different places $x_1$, $x_2$, then $\int_{(x_1,N)}^{(x_2,N)} \alpha_0
= \tilde \beta(x_2)- \tilde \beta(x_1)=0$, and $\int_{(x_1,N)}^{(x_2,N)} \alpha =
\mu(x_2-x_1)$.

{\bf Step 3:} Now we use the fact that $\alpha$ is an integral class,
and is the pullback of a class on an approximant that describes the
tiling out to a distance equal to the pattern-equivariance radius
$r$. Thus, for any patch of size greater than $r$, the integral of
$\alpha$ from one occurrence of the patch to another occurrence of the
same patch must be an integer.  In the setting of the previous step,
$\mu(x_2-x_1) \in \Z$.

{\bf Step 4:} There exists a value of $m$ such that $P_0$ appears in
all $m$-supertiles of $\Omega_{\sigma_{\mathbf{n}}}$, and also in all
supertiles of order greater than $m$. Since the word $cc$ appears in
the language of $\sigma_{\mathbf{n}}$, one can find $P_0$ in
corresponding locations of $m$-supertiles. That is, we can take
$x_2-x_1$ to be the length of an $m$-supertile of type $c$, or the
length of an $(m+1)$-supertile of type $c$ (which is the same as the
length of an $m$-supertile of type $a$). Since a tile in
$\Omega_{\sigma_{\mathbf{n}}}$ is actually an $n$-supertile in
$\Omega_F$, these are the widths $W_c$ and $W_a$ of $(m+n)$-supertiles
of type $c$ and $a$, respectively.  The upshot is that $\mu W_c$ and
$\mu W_a$ are integers.

We have already shown that $W_c$ and $W_a$ are relatively prime, so
there are integers $j$ and $k$ such that $1 = jW_c + kW_a$. But then
$\mu = j(\mu W_c) + k(\mu W_a)$ is an integer.

{\bf Step 5:} The involution of the fusion rule (\ref{DPV-fusion})
extends to an involution of $\Omega_F$ itself. Let $[\alpha']$ be the
pullback of $[\alpha]$ by this involution.  The Ruelle-Sullivan map
sends $[\alpha']$ to $(\nu, \mu)$. The previous arguments, applied to
$[\alpha']$, then show that $\nu \in \Z$.

\end{proof}

Theorem \ref{mainthm} says that the image of $H^1(\Omega_F)$ under the
Ruelle-Sullivan map is $\Z^2$. When it comes to the Kellendonk-Putnam and
Giordano-Putnam-Skau conjectures, to virtual eigenvalues, and to
possible bundle structures, that is the important result. However, it
is possible to say more.

\begin{thm} \label{icing-on-the-cake} $H^1(\Omega_F) = \Z^2$, and is
  generated by the classes of the constant forms $dx$ and $dy$.
 \end{thm}

 \begin{proof}  We must show that $[\alpha] = \mu [dx] + \nu[dy]$. That is, we must show that, for any two occurrences
 of any sufficiently large patch $P$, say at positions $(x_1,y_1)$ and $(x_2,y_2)$,
 \begin{equation} \label{integer-integral}
 \int_{(x_1,y_1)}^{(x_2,y_2)} \alpha = \mu(x_2-x_1) + \nu (y_2-y_1).
 \end{equation}
 By primitivity, we can restrict attention to the case that $P$ is an $n$-supertile (for some sufficiently large value of $n$) of
 type $d$.

 Our previous arguments show that equation (\ref{integer-integral})
 holds whenever $y_2=y_1$, when the path of integration is horizontal,
 and when all the supertiles that appear on the path are aligned on
 their bottom edges. Similar arguments, involving the topology of
 $\Omega_{\sigma_{\mathbf{s}}}$, apply when the path is horizontal and
 the supertiles are aligned along their top edges.  Likewise, the
 equation applies when the path is vertical and successive supertiles
 are aligned on either their right or left edges.  To complete the
 proof, we must show that there is a path from $(x_1,y_1)$ to
 $(x_2,y_2)$ in which successive supertiles are aligned on one side or
 another.

 This follows from the structure of the fusion rule
 (\ref{DPV-fusion}). All of the $n$-supertiles in an $(n+1)$ supertile
 of type $b$, $c$ or $d$ are aligned. Within an $(n+1)$-supertile of
 type $a$, all of the $n$-supertiles on the top row are aligned, all
 of the supertiles on the right column are aligned, and the
 nine-remaining supertiles are aligned. Since the upper right
 $n$-supertile is aligned with both the right column and the top row,
 and since the upper left supertile is aligned with both the top row
 and some of the elements of the lower left $3 \times 3$ block, it is
 possible to get from any $n$-supertile to any other within the
 $(n+1)$-supertile by following aligned edges.

\begin{center}
 \includegraphics[width=0.4\textwidth]{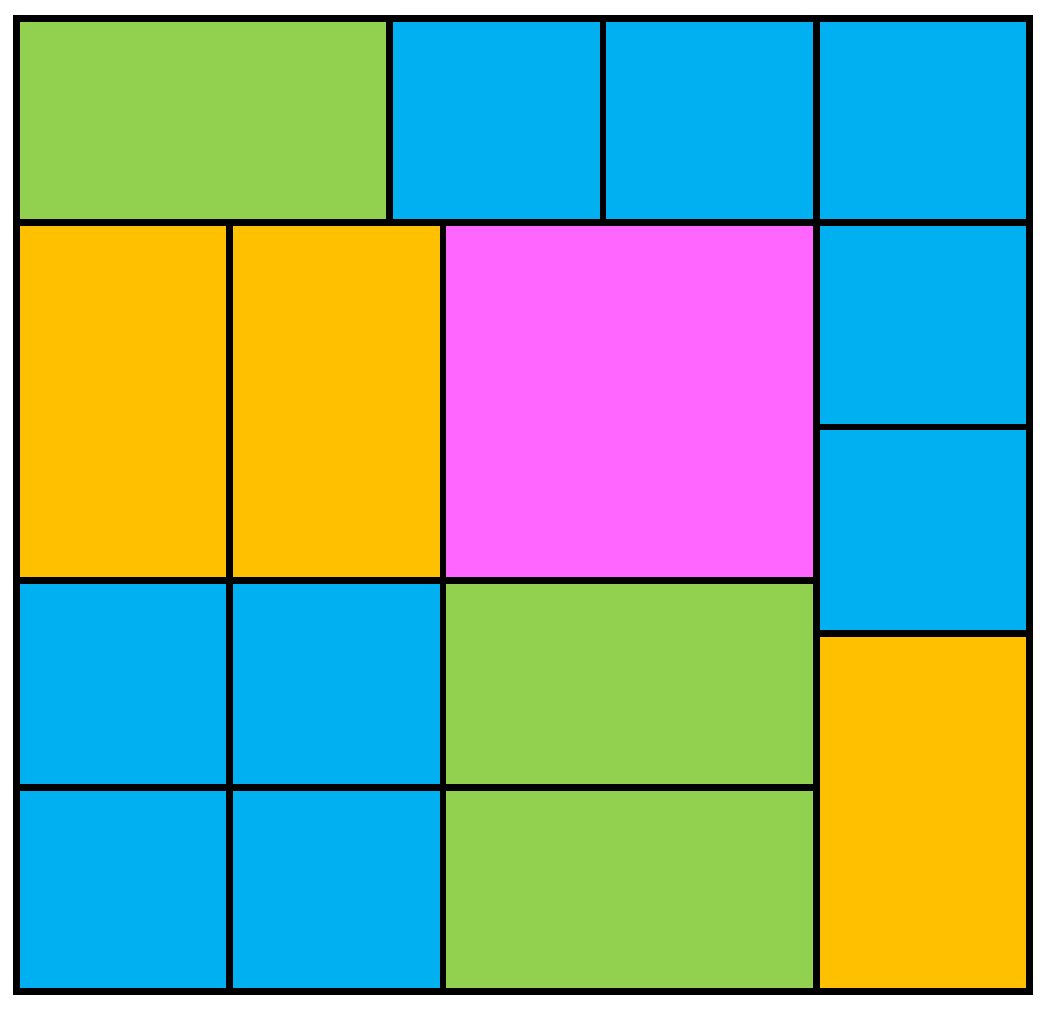} \hspace{.2in}
\end{center}

\begin{center}
{\small The $n$-supertiles within $P_{n+1}(a)$}
\end{center}

That $(n+1)$-supertile sits inside of an $(n+2)$-supertile.  Once at
the boundary of the $(n+1)$-supertile, one can get to any other
$(n+1)$ supertile in the same $(n+2)$-supertile by following edges of
$(n+1)$-supertiles. By induction, if an $n$-supertile sits inside of
an $n'$-supertile, for any $n'>n$, it is possible to go from the given
$n$-supertile to the boundary of the $n'$-supertile by following
aligned $n$-supertiles.

\begin{center}
 \includegraphics[width=0.6\textwidth]{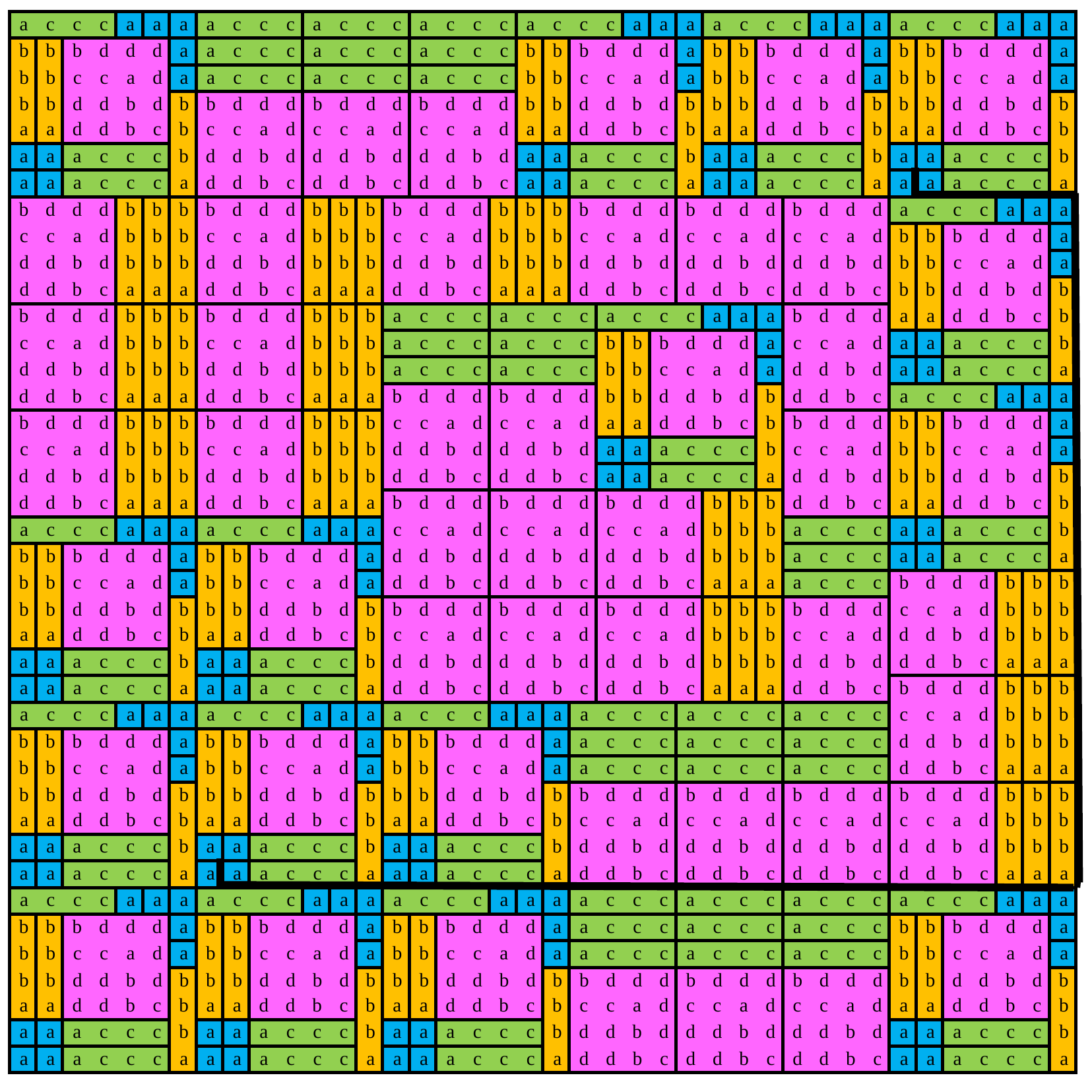} \hspace{.2in}
\end{center}

\begin{center}
{\small A path, in bold, along aligned  $1$-supertiles within $P_{4}(a)$}
\end{center}

If the two occurrences of $P$ lie in the same $n'$ supertile of some
order, it is thus possible to go from either one to a common boundary
point, and hence to go from one to the other, only following aligned
$n$-supertiles. On a generic tiling, this is always the case, in that
the only infinite-order supertile is the entire plane. Since equation
(\ref{integer-integral}) applies for every pair of patches in a
generic tiling, and since minimality implies that the
pattern-equivariant cohomology of every tiling is the same, equation
(\ref{integer-integral}) applies to all pairs of patches in all
tilings.

 \end{proof}

 It is worth noting three key features of Frank's DPV tiling that make the
 proofs of Theorems \ref{mainthm} and \ref{icing-on-the-cake}
 work. These features, or features that work just as well, are to be
 expected from a large class of FLC tilings of $\R^2$ with shears in
 both directions.  In fact, there are no known examples of tilings
 with shears that do {\em not} have these properties, giving support
 for Conjecture \ref{shears}.

\begin{enumerate}
\item In Frank's DPV tiling, the 1-dimensional tiling space along the
  lower boundary of supertiles had no asymptotically negligible (AN)
  classes. This was very convenient, but in cases where AN classes do
  exist, we can fall back on the following argument:

  It is impossible for an integral class to be asymptotically
  negligible without being trivial, since asymptotically negligible
  means that integrals are determined locally up to $\epsilon$, while
  integrality leaves no room for small uncertainty. Thus it is
  impossible to have an integral class that is a rational multiple of
  length plus something asymptotically negligible.

  Next, we must show that an integral class cannot be an {\em
    irrational} multiple of length plus an AN class. Equivalently, we
  must show that we cannot have two integral classes whose images in
  the so-called mixed cohomology (that is, $H^1$ mod AN) are
  irrational multiples of one another.

  This depends on number-theoretic properties of the stretching
  factor, specifically on the size of the Galois group of its
  splitting field.  We do not know whether our desired property is
  always true, but it is certainly true whenever the characteristic
  polynomial has more eigenvalues of magnitude 1 or bigger than of
  magnitude less than 1, and in particular is true whenever the
  stretching factor is a non-Pisot quadratic or cubic.
  Note that having a non-Pisot stretching factor was needed to get
  shears in the first place \cite{FR}, and that all constructions to date
  of tilings with shears have involved quadratic or cubic stretching
  factors.

\item The lengths of the supertiles have no common factor.  Again,
  this is a necessary condition for having shears.  If the lengths of
  all $n$-supertiles had a common factor, then the offsets between
  adjacent $n$-supertiles would have to be multiples of that common
  factor.

\item It was possible to find a path from any $n$-supertile to any
  other along aligned $n$-supertiles.  This is a general feature of
  DPV tilings, which are set up to be products of 1-dimensional
  fusions, only with a block (in this case the lower left $3 \times 3$
  block of $P_{n+1}(a)$) rotated or reflected. The block remains
  aligned, the rest of the supertile remains aligned, and it is
  possible to go from the block to the rest at a corner of the block.
\end{enumerate}

In summary, the Frank DPV tiling is
exceptionally simple to work with, and allows us to prove
Theorem \ref{nogo} directly. Other tilings with shears, constructed from
DPV's with non-Pisot stretching factors of degree 2 or 3,
are likewise counterexamples
to the conjectures of Giordano, Putnam and Skau. The question of whether
{\em all} subshifts with shears are counterexamples (c.f. Conjecture
\ref{shears}) remains open.

\begin{ex}
Our main counterexample had shears in both directions and $H^1=\Z^2$. The
following example, taken from \cite{FrankSadunCoho}, is a tiling space with
shears in only one direction. It admits small cocycles, but
the image of the Ruelle-Sullivan map is not dense.

Consider a 2-dimensional tiling with two tile types, $a$ and $b$, both of
which are unit squares. These form the basis of a fusion tiling with the
rule
$$ P_{n+1}(a) = \left [ \begin{matrix} P_n(a)&P_n(b)\cr P_n(b)&P_n(a)
\end{matrix} \right ], \quad
P_{n+1}(b) = \left [ \begin{matrix} P_n(a) & P_n(a) & P_n( a) \cr
    P_n(a) & P_n(a) & P_n(a) \end{matrix} \right ]. $$
As with the previous example, there is a horizontal shear.

Restricting to a single row, the substitution $a \to ab$ (or $ba$), $b
\to aaa$ has eigenvalues $(1 \pm \sqrt{13})/2$, both of which are
bigger than 1. The second eigenvalue controls the discrepancy in how
many $m$-supertiles of type $a$ (versus $b$) appear on either side of a
horizontal boundary between $n$-supertiles, where $n \gg m$. This discrepancy is
unbounded, growing as $|(1-\sqrt{13})/2|^n$,
and the lengths of $P_m(a)$ and $P_m(b)$ are relatively
prime, so $m$-supertiles meet with arbitrary offsets.  Taking a limit
as $m \to \infty$ proves that this tiling admits horizontal shears.

Also exactly as before,
the horizontal part of a closed PE cochain $\alpha$ has to be an integer
multiple of length, plus something exact. This is because the first
cohomology of the 1-dimensional tiling space with substitution $a \to ba$,
$b \to aaa$ has no asymptotically negligible classes, insofar as the
eigenvalues of $\left ( \begin{smallmatrix} 1 & 3 \cr 1 & 0 \end{smallmatrix}
\right )$ are both larger than 1.
Thus, if $(\lambda_x, \lambda_y)$
is in the image of the Ruelle-Sullivan map, $\lambda_x$ must be an integer.

However, $\lambda_y$ can be arbitrarily small. For each $n$, we
can construct a class that essentially counts $n$-supertiles in the vertical
direction. This class can be represented by a 1-cochain $\alpha$ that evaluates
to 0 on all horizontal edges, to 1 on the vertical edges of the bottom row of
each $n$-supertile, and to 0 on all other horizontal edges. Since there are
no vertical shears, these bottom rows can be identified in a strongly PE
manner. The Ruelle-Sullivan map sends this class to $(0, 2^{-n})$.

Since the image of RS contains small elements but is not dense, we can
write $\Omega$ as a bundle over tori with large volume, but not over tori
that are large in both directions.

The corresponding subshift does not admit small, positive cocycles with
respect to the obvious coordinates, since $(0,2^{-n})$ is not a positive
vector on the first quadrant. However, if we rotate our axes by 45 degrees
(i.e., restrict our $\Z^2$ action to combinations of $(1,1)$ and $(1,-1)$),
we obtain a $\Z^2$ action that does admit small, positive cocycles.
This implies the existence of small equivalence
relations. However, those relations were already manifest from the hierarchical
structure of the tiling.
\end{ex}


\end{document}